\documentclass[a4paper,12pt]{article}

\usepackage{amsmath}
\usepackage{amsthm}
\usepackage{enumerate}
\usepackage{amssymb,amsfonts,latexsym,mathtools}
\usepackage{mathrsfs}
\usepackage{bm}
\usepackage{cases}
\usepackage{color}

\newcommand{\Levy}{L\'{e}vy}

\newcommand{\R}{\mathbb{R}}
\newcommand{\F}{\mathscr{F}}

\newcommand{\e}{\varepsilon}

\newcommand{\cadlag}{c\`{a}dl\`{a}g}
\renewcommand{\P}{\mathbb{P}}

\numberwithin{equation}{section}

\makeatletter
\renewcommand\section{\@startsection {section}{1}{\z@}%
{-3.5ex \@plus -1ex \@minus -.2ex}%
{2.3ex \@plus.2ex}%
{\normalfont\large\bf}}
\makeatother

\makeatletter
\renewcommand\subsection{\@startsection {subsection}{1}{\z@}%
{-3.5ex \@plus -1ex \@minus -.2ex}%
{2.3ex \@plus.2ex}%
{\normalfont\normalsize\bf}}
\makeatother

\theoremstyle{plain}
\newtheorem{thm}{Theorem}[section]

\newtheorem{prop}[thm]{Proposition}
\newtheorem*{thm*}{Theorem}

\theoremstyle{definition}

\newtheorem*{Rem*}{Remark}
\allowdisplaybreaks

\usepackage[dvipdfmx]{hyperref}
\hypersetup{
setpagesize=false,
 bookmarksnumbered=true,
 bookmarksopen=true,
 colorlinks=true,
 linkcolor=blue,
 citecolor=red,
}

\begin{document}
\begin{center}
\Large \textbf{Conditionings to avoid points with various clocks for \Levy\ processes}
\end{center}
\begin{center}
Kohki IBA (Graduate School of Science, Osaka University)\\
Kouji YANO (Graduate School of Science, Osaka University)
\end{center}

\begin{abstract}
We discuss conditionings to avoid two points and one-point local time penalizations with conditioning to avoid another point, for which we adopt various clocks. We also give corrections to some of the previous results of Takeda--Yano for one-point local time penalizations.
\end{abstract}

\section{Introduction}
A penalization problem is to study the long-time limit of the form
\begin{align}
\label{1}
\lim_{\tau\to \infty}\frac{\mathbb{P}_x[F_s\cdot\Gamma_\tau]}{\mathbb{P}_x[\Gamma_\tau]},
\end{align}
where $(X=(X_t)_{t\ge 0},(\F_t)_{t\ge 0},(\mathbb{P}_x)_{x\in \R})$ is a Markov process, $(\Gamma_t)_{t\ge 0}$ is a non-negative process called a \emph{weight}, $(F_s)_{s\ge 0}$ is a process of test functions adapted to $(\F_s)_{s\ge 0}$, and $\tau$ is a net of parametrized random times tending to infinity, called a \emph{clock}.

We follow the notations and adopt the assumptions of Iba--Yano \cite{ISY} (see Section \ref{S2} for the details). In \cite{ISY}, we considered the case where the weight process is given by
\begin{align}
\label{a-1}
\Gamma_{a,b,t}^{\lambda_a,\lambda_b}:=e^{-\lambda_aL_t^a-\lambda_bL_t^b}.
\end{align}
In this paper, we consider the case where the weight process is given by
\begin{align}
\Gamma_{a,b,t}^{\lambda_a,\infty}:&=\lim_{\lambda_b\to \infty}\Gamma_{a,b,t}^{\lambda_a,\lambda_b}=e^{-\lambda_a L_t^a}\cdot 1_{\{L_t^b=0\}},\\
\Gamma_{a,b,t}^{\infty,\infty}:&=\lim_{\lambda_a,\lambda_b\to \infty}\Gamma_{a,b,t}^{\lambda_a,\lambda_b}=\lim_{\lambda_a\to \infty}\Gamma_{a,b,t}^{\lambda_a,\infty}=1_{\{L_t^a=L_t^b=0\}},
\end{align}
where the limit processes may be formally obtained by the limit as $\lambda_a\to \infty$ or $\lambda_b\to \infty$ of those of (\ref{a-1}), although the penalization limit should be proven independently.

\subsection*{Two-point local time penalization}
For $-1\le \gamma\le 1$, we say\footnote{To describe the penalization limits, our limit $(c,d) \stackrel{(\gamma)}{\to}\infty$ is more suitable than the limit $(c,d)\stackrel{\gamma}{\to}\infty$ of the equation (1.9) of Takeda--Yano \cite{TY}. See, Section \ref{S6}.}
\begin{align}
(c,d)\stackrel{(\gamma)}{\to}\infty\ \mathrm{when}\ c\to \infty,\ d\to \infty,\ \mathrm{and}\ \frac{d-c}{c+d}\to \gamma.
\end{align}
Here for the random clock $\tau=(\tau_\lambda)$, we adopt one of the following:
\begin{enumerate}
\item exponential\ clock: $\tau=(\bm{e}_q)$ as $q\to 0+$, where $\bm{e}_q$ has the exponential distribution with its parameter $q>0$ and is independent of $X$;
\item hitting\ time\ clock: $\tau=(T_c)$ as $c\to \pm \infty$, where $T_c$ is the first hitting time at $c$;
\item two-point hitting time clock: $\tau=(T_c\wedge T_{-d})$ as $(c,d)\stackrel{(\gamma)}{\to}\infty$;
\item inverse\ local\ time\ clock: $\tau=(\eta_u^c)$ as $c\to \pm \infty$, where $\eta^c=(\eta_u^c)_{u\ge 0}$ is an inverse local time.
\end{enumerate}
Then, Iba--Yano \cite{ISY} showed the following:
\begin{thm}[Iba--Yano \cite{ISY}]
For distinct points $a,b\in \R$, for constants $\lambda_a,\lambda_b>0$, and for a constant $-1\le \gamma \le 1$, there exists a positive function $\varphi_{a,b}^{(\gamma),\lambda_a,\lambda_b}(x)$ such that the process
\begin{align}
\label{0-1}
\left(M_{a,b,s}^{(\gamma),\lambda_a,\lambda_b}:=\varphi_{a,b}^{(\gamma),\lambda_a,\lambda_b}(X_s)e^{-\lambda_aL_t^a-\lambda_b L_{t}^b}\right)_{s\ge 0}
\end{align}
is a martingale, and the following assertions hold:
\begin{enumerate}
\item exponential clock: $\displaystyle\lim_{q\to 0+}r_q(0)\mathbb{P}_x\left[F_s\cdot \Gamma_{a,b,\bm{e}_q}^{\lambda_a,\lambda_b}\right]= \mathbb{P}_x\left[F_s \cdot M_{a,b,s}^{(0),\lambda_a,\lambda_b}\right],$
\item hitting time clock: $\displaystyle \lim_{c\to \pm \infty }h^B(c)\mathbb{P}_x\left[F_s\cdot \Gamma_{a,b,T_c}^{\lambda_a,\lambda_b}\right]=\mathbb{P}_x\left[F_s \cdot M_{a,b,s}^{(\pm 1),\lambda_a,\lambda_b}\right],$
\item two-point hitting time clock: $\displaystyle \lim_{(c,d)\stackrel{(\gamma)}{\to}\infty}h^C(c,-d)\mathbb{P}_x\left[F_s\cdot \Gamma_{a,b,T_c\wedge T_{-d}}^{\lambda_a,\lambda_b}\right]= \mathbb{P}_x\left[F_s \cdot M_{a,b,s}^{(\gamma),\lambda_a,\lambda_b}\right],$
\item inverse local time clock: $\displaystyle  \lim_{c\to \pm \infty}h^B(c)\mathbb{P}_x\left[F_s\cdot \Gamma_{a,b,\eta_u^c}^{\lambda_a,\lambda_b}\right]=\mathbb{P}_x\left[F_s \cdot M_{a,b,s}^{(\pm 1),\lambda_a,\lambda_b}\right],$
\end{enumerate}
where $r_q$ is the $q$-resolvent density (see (\ref{b7})), $h^B$ is defined by (\ref{b34}), and $h^C$ is defined by (\ref{b46}).
\end{thm}
We write for $-1\le \gamma\le 1$,
\begin{align*}
h^{(\gamma)}(x):=h(x)+\frac{\gamma}{m^2}x,
\end{align*}
where $m^2=\mathbb{P}_0[X_1^2]$ and the function $h$ is a renormalized zero resolvent (see Proposition \ref{b28}). Then, the function $\varphi_{a,b}^{(\gamma),\lambda_a,\lambda_b}$ is given explicitly as follows:
\begin{align*}
\label{c41}
\varphi_{a,b}^{(\gamma),\lambda_a,\lambda_b}(x)&=h^{(\gamma)}(x-a)-\mathbb{P}_x(T_b<T_a)h^{(\gamma)}(b-a)\\
&\qquad +\mathbb{P}_x(T_a<T_b)\cdot \frac{h^{(\gamma)}(a-b)}{1+\lambda_ah^B(a-b)}\\
&\qquad +\mathbb{P}_x(T_a<T_b)\cdot\frac{1}{1+\lambda_ah^B(b-a)}\cdot \frac{1+\lambda_a h^{(\gamma)}(b-a)}{\lambda_a+\lambda_b+\lambda_a\lambda_b h^B(a-b)}\\
&\qquad +\mathbb{P}_x(T_b<T_a)\cdot \frac{h^{(\gamma)}(b-a)}{1+\lambda_bh^B(a-b)}\\
&\qquad +\mathbb{P}_x(T_b<T_a)\cdot\frac{1}{1+\lambda_bh^B(b-a)}\cdot \frac{1+\lambda_b h^{(\gamma)}(a-b)}{\lambda_a+\lambda_b+\lambda_a\lambda_b h^B(a-b)}.
\stepcounter{equation}\tag{\theequation}
\end{align*}
Note that $\varphi_{a,b}^{(\gamma),\lambda_a,\lambda_b}(x)$ is symmetric with respect to $a$ and $b$, i.e., for $x\in \R,$
\begin{align}
\varphi_{a,b}^{(\gamma),\lambda_a,\lambda_b}(x)=\varphi_{b,a}^{(\gamma),\lambda_a,\lambda_b}(x).
\end{align}
When $m^2=\infty$, we have
\begin{align}
\label{d41}
\varphi_{a,b}^{(\gamma),\lambda_a,\lambda_b}(x)\equiv \varphi_{a,b}^{(0),\lambda_a,\lambda_b}(x).
\end{align}
Therefore, when $m^2=\infty,$ we obtain
\begin{align}
\label{d44-1}
M_{a,b,t}^{(\gamma),\lambda_a,\lambda_b}=M_{a,b,t}^{(0),\lambda_a,\lambda_b}
\end{align}
for any $-1\le \gamma\le 1.$

\subsection*{Main theorems about conditionings}
We define
\begin{align}
\varphi_{a,b}^{(\gamma),\lambda_a,\infty}(x):&=h^{(\gamma)}(x-a)-\mathbb{P}_x(T_b<T_a)h^{(\gamma)}(b-a)+\mathbb{P}_x(T_a<T_b)\cdot \frac{h^{(\gamma)}(a-b)}{1+\lambda_ah^B(a-b)},\\
\varphi_{a,b}^{(\gamma),\infty,\infty}(x):&=h^{(\gamma)}(x-a)-\mathbb{P}_x(T_b<T_a)h^{(\gamma)}(b-a).
\end{align}
When $m^2=\infty$, we have
\begin{align}
\varphi_{a,b}^{(\gamma),\lambda_a,\infty}(x)&\equiv \varphi_{a,b}^{(0),\lambda_a,\infty}(x),\\
\varphi_{a,b}^{(\gamma),\infty,\infty}(x)&\equiv \varphi_{a,b}^{(0),\infty,\infty}(x),
\end{align}
Note that $\varphi_{a,b}^{(\gamma),\infty,\infty}$ and $\varphi_{a,b}^{(\gamma),\lambda_a,\infty}$ are formally obtained by the limit as $\lambda_a\to \infty$ or $\lambda_b\to \infty$ of those of (\ref{c41}), i.e., 
\begin{align}
\varphi_{a,b}^{(\gamma),\lambda_a,\infty}(x)&=\lim_{\lambda_b\to \infty}\varphi_{a,b}^{(\gamma),\lambda_a,\lambda_b}(x),\\
\varphi_{a,b}^{(\gamma),\infty,\infty}(x)&=\lim_{\lambda_a,\lambda_b\to \infty}\varphi_{a,b}^{(\gamma),\lambda_a,\lambda_b}(x)=\lim_{\lambda_a\to \infty}\varphi_{a,b}^{(\gamma),\lambda_a,\infty}(x).
\end{align}

Under the same assumption, our main theorems are as follows (see Sections \ref{S3} and \ref{S4} for the details):

\begin{thm}
For distinct points $a,b\in \R$, for a constant $\lambda_a>0$, and for a constant $-1\le \gamma \le 1$, the process
\begin{align}
\left(M_{a,b,s}^{(\gamma),\lambda_a,\infty}:=\varphi_{a,b}^{(\gamma),\lambda_a,\infty}(X_s)e^{-\lambda_aL_t^a}1_{\{L_s^b=0\}}\right)_{s\ge 0}
\end{align}
is a martingale, and the following assertions hold:
\begin{enumerate}
\item exponential clock: $\displaystyle\lim_{q\to 0+}r_q(0)\mathbb{P}_x\left[F_s\cdot \Gamma_{a,b,\bm{e}_q}^{\lambda_a,\infty}\right]= \mathbb{P}_x\left[F_s \cdot M_{a,b,s}^{(0),\lambda_a,\infty}\right],$
\item hitting time clock: $\displaystyle \lim_{c\to \pm \infty }h^B(c)\mathbb{P}_x\left[F_s\cdot \Gamma_{a,b,T_c}^{\lambda_a,\infty}\right]=\mathbb{P}_x\left[F_s \cdot M_{a,b,s}^{(\pm 1),\lambda_a,\infty}\right],$
\item two-point hitting time clock: $\displaystyle \lim_{(c,d)\stackrel{(\gamma)}{\to}\infty}h^C(c,-d)\mathbb{P}_x\left[F_s\cdot \Gamma_{a,b,T_c\wedge T_{-d}}^{\lambda_a,\infty}\right]= \mathbb{P}_x\left[F_s \cdot M_{a,b,s}^{(\gamma),\lambda_a,\infty}\right],$
\item inverse local time clock: $\displaystyle  \lim_{c\to \pm \infty}h^B(c)\mathbb{P}_x\left[F_s\cdot \Gamma_{a,b,\eta_u^c}^{\lambda_a,\infty}\right]=\mathbb{P}_x\left[F_s \cdot M_{a,b,s}^{(\pm 1),\lambda_a,\infty}\right].$
\end{enumerate}
\end{thm}

\begin{thm}
For distinct points $a,b\in \R$ and for a constant $-1\le \gamma \le 1$, the process
\begin{align}
\left(M_{a,b,s}^{(\gamma),\infty,\infty}:=\varphi_{a,b}^{(\gamma),\infty,\infty}(X_s)1_{\{L_s^a=L_s^b=0\}}\right)_{s\ge 0}
\end{align}
is a martingale, and the following assertions hold:
\begin{enumerate}
\item exponential clock: $\displaystyle\lim_{q\to 0+}r_q(0)\mathbb{P}_x\left[F_s\cdot \Gamma_{a,b,\bm{e}_q}^{\infty,\infty}\right]= \mathbb{P}_x\left[F_s \cdot M_{a,b,s}^{(0),\infty,\infty}\right],$
\item hitting time clock: $\displaystyle \lim_{c\to \pm \infty }h^B(c)\mathbb{P}_x\left[F_s\cdot \Gamma_{a,b,T_c}^{\infty,\infty}\right]=\mathbb{P}_x\left[F_s \cdot M_{a,b,s}^{(\pm 1),\infty,\infty}\right],$
\item two-point hitting time clock: $\displaystyle \lim_{(c,d)\stackrel{(\gamma)}{\to}\infty}h^C(c,-d)\mathbb{P}_x\left[F_s\cdot \Gamma_{a,b,T_c\wedge T_{-d}}^{\infty,\infty}\right]= \mathbb{P}_x\left[F_s \cdot M_{a,b,s}^{(\gamma),\infty,\infty}\right],$
\item inverse local time clock: $\displaystyle  \lim_{c\to \pm \infty}h^B(c)\mathbb{P}_x\left[F_s\cdot \Gamma_{a,b,\eta_u^c}^{\infty,\infty}\right]=\mathbb{P}_x\left[F_s \cdot M_{a,b,s}^{(\pm 1),\infty,\infty}\right].$
\end{enumerate}
\end{thm}

\subsection*{Organization}
This paper is organized as follows. In Section \ref{S2}, we prepare some general results of
\Levy\ processes. In Sections \ref{S3} and \ref{S4}, we discuss the penalization results with $\Gamma_{a,b,t}^{\infty,\infty}$ and $\Gamma_{a,b,t}^{\lambda_a,\infty}$, respectively. In Section \ref{S5}, we discuss the penalization results with inverse local time clock. In Section \ref{S6}, we give some corrections to some results of Takeda--Yano \cite{TY}.


\section{Preliminaries}
\label{S2}
\subsection{\Levy\ process and resolvent density}
Let $(X,\mathbb{P}_x)$ be the canonical representation of a real valued \Levy\ process starting from $x\in \R$ on the \cadlag\ path space. For $t>0$, we denote by $\F_t^X=\sigma(X_s,\ 0\le s\le t)$ the natural filtration of $X$ and write $\F_t=\bigcap_{s>t}\F_s^X$. For $a\in \R$, let $T_a$ be the hitting time of $\{a\}$ for $X$, i.e.,
\begin{align}
\label{b1}
T_a=\inf\{t>0:\ X_t=a\}.
\end{align}
For $\lambda\in \R$, we denote by $\Psi(\lambda)$ the characteristic exponent of $X$, which satisfies 
\begin{align}
\label{b2}
\mathbb{P}_0\left[e^{i\lambda X_t}\right]=e^{-t\Psi(\lambda)}
\end{align}
for $t\ge 0.$ Moreover, by \Levy--Khinchin formula, it is denoted by
\begin{align}
\label{b3}
\Psi(\lambda)=iv\lambda+\frac{1}{2}\sigma^2 \lambda^2+\int_{\R}\left(1-e^{i\lambda x}+i \lambda x 1_{\{|x|<1\}}\right)\nu(dx),
\end{align}
where $v\in \R,\ \sigma\ge 0,\ $and $\nu$ is a measure on $\R$, called a \emph{\Levy\ measure}, with $\nu(\{0\})=0$ and $\int_{\R}(x^2\wedge 1)\nu(dx)<\infty.$

Throughout this paper, we always assume $(X,\mathbb{P}_0)$ is recurrent, i.e., 
\begin{align}
\label{b4}
\mathbb{P}_0\left[\int_0^\infty 1_{\{|X_t-a|<\e\}}dt\right]=\infty
\end{align}
for all $a\in \R$ and $\e>0$, and always assume the condition
\begin{align*}
\textbf{(A)}\ \int_0^\infty \left|\frac{1}{q+\Psi(\lambda)}\right|d\lambda<\infty\qquad \mathrm{for\ each}\ q>0.
\end{align*}
It is known that $X$ has a bounded continuous resolvent density $r_q$:
\begin{align}
\label{b7}
\int_{\R}f(x)r_q(x)dx=\mathbb{P}_0\left[\int_0^\infty e^{-qt}f(X_t)dt\right]
\end{align}
holds for $q>0$ and non-negative measurable functions $f$. See, e.g., Theorems II.16 and II.19 of \cite{Ber}. Moreover, there exists an equality that connects the hitting time of $0$ and the resolvent density:
\begin{align}
\label{b9}
\mathbb{P}_x\left[e^{-qT_0}\right]=\frac{r_q(-x)}{r_q(0)}
\end{align}
for $q>0$ and $x\in \R.$ See, e.g., Corollary II.18 of \cite{Ber}.


\subsection{Local time and excursion}
We denote by $\mathcal{D}$ the set of \cadlag\ paths $e:[0,\infty)\to \R\cup\{\Delta\}$ such that
\begin{align}
\label{b12}
\begin{cases}
e(t)\in \R\setminus \{0\}&\mathrm{for}\ 0<t<\zeta(e),\\
e(t)=\Delta &\mathrm{for}\ t\ge \zeta(e),
\end{cases}
\end{align}
where the point $\Delta$ is an isolated point and $\zeta$ is the excursion length, i.e.,
\begin{align}
\label{b13}
\zeta=\zeta(e):=\inf \{t>0:\ e(t)=\Delta\}.
\end{align}
Let $\Sigma$ denote the $\sigma$-algebra on $\mathcal{D}$ generated by cylinder sets.

Assume the condition \textbf{(A)} holds. Then, we can define a local time at $a\in \R$, which we denote by $L^a=(L_t^a)_{t\ge 0}$. It is defined by
\begin{align}
\label{b14}
L_t^a:=\lim_{\e\to 0+}\frac{1}{2\e}\int_0^t1_{\{|X_s-a|<\e\}}ds.
\end{align}
It is known that $L^a$ is continuous in $t$ and satisfies 
\begin{align}
\label{b15}
\mathbb{P}_x\left[\int_0^\infty e^{-qt}dL_t^a\right]=r_q(a-x)
\end{align}
for $q>0$ and $x\in \R$. See, e.g., Section V of \cite{Ber}. In particular, from this expression, $r_q(x)$ is non-decreasing as $q\to 0+.$

Let $\eta^a=(\eta_l^a)_{l\ge 0}$ be an inverse local time, i.e.,
\begin{align}
\label{b16}
\eta_l^a:=\inf\{t>0:\ L_t^a>l\}.
\end{align}
It is known that the process $(\eta^a,\mathbb{P}_a)$ is a possibly killed subordinator which has the Laplace exponent
\begin{align}
\label{b17}
\mathbb{P}_a\left[e^{-q\eta_l^a}\right]=e^{-\frac{l}{r_q(0)}}
\end{align}
for $l>0$ and $q>0$. See, e.g., Proposition V.4 of \cite{Ber}.

We denote $\epsilon_l^a$ for an excursion away from $a\in \R$ which starts at local time $l\ge 0$, i.e.,
\begin{align}
\label{b18}
\epsilon_l^a(t):=\begin{cases}
X_{t+\eta_{l-}^a}&\mathrm{for}\ 0\le t<\eta_{l}^a-\eta_{l-}^a,\\
\Delta&\mathrm{for}\ t\ge \eta_l^a-\eta_{l-}^a.
\end{cases}
\end{align}
Then, $(\epsilon_l^a)_{l\ge 0}$ is a Poisson point process, and we write $n^a$ for the characteristic measure of $\epsilon^a$. It is known that $(\mathcal{D},\Sigma,n^a)$ is a $\sigma$-finite measure space. See, e.g., Section IV of \cite{Ber}. For $B\in \mathscr{B}(0,\infty)\otimes \Sigma$, we define
\begin{align}
\label{b19}
N^a(B):=\#\{(l,e)\in B:\ \epsilon_l^a=e\}.
\end{align}
Then, $N^a$ is a Poisson random measure with its intensity measure $ds\times n^a(de).$ It is known that the subordinator $\eta^0$ has no drift and its \Levy\ measure is $n^0(T_0\in dx)$.


\subsection{The renormalized zero resolvent}
We define
\begin{align}
\label{b27}
h_q(x):=r_q(0)-r_q(-x)
\end{align}
for $q>0$ and $x\in \R.$ It is clear that $h_q(0)=0$, and by (\ref{b9}), we have $h_q(x)\ge 0$. The following theorem plays a key role in our penalization results.
\begin{prop}[Theorem 1.1 of \cite{TY}]
\label{b28}
If the condition $\textbf{(A)}$ is satisfied, then for any $x\in \R$, $ h(x):=\lim_{q\to 0+}h_q(x)$ exists and is finite.
\end{prop}
We call $h$ the \emph{renormalized zero resolvent}. Moreover, we introduce the functions $h_q^B$ and $h^B.$
\begin{prop}[Lemma 3.5 of \cite{TY}]
\label{b32}
For $a\in \R,$ it holds that
\begin{align}
\label{b34}
h^B(a):=\lim_{q\to 0+}h_q^B(a)=\mathbb{P}_0[L_{T_a}]=h(a)+h(-a).
\end{align}
\end{prop}

\begin{prop}[Lemma 6.1 of \cite{TY}\footnote{In Takeda--Yano \cite{TY}, there is an error in the assertion of Lemma 6.1. Proposition 2.3 corrects that error. See, Section \ref{S6}.}]
\label{b45}
For distinct points $a,b\in \R$, it holds that
\begin{align*}
\label{b46}
h^C(a,b):&=\mathbb{P}_0[L_{T_a\wedge T_b}^0]\\
&=\frac{1}{h^B(a-b)}\left\{\begin{aligned}
&(h(b)+h(-a))h(a-b)+(h(a)+h(-b))h(b-a)\\
&\qquad +(h(a)-h(b))(h(-b)-h(-a))-h(a-b)h(b-a)
\end{aligned}\right\}.
\stepcounter{equation}\tag{\theequation}
\end{align*}
\end{prop}


\section{Conditionings to avoid two points}
\label{S3}
We define
\begin{align}
\label{c61}
N_{a,b,t}^{q,\lambda_a,\lambda_b}:&=\displaystyle r_q(0)\mathbb{P}_x\left[\Gamma_{a,b,\bm{e}_q}^{\lambda_a,\lambda_b};\ t<\bm{e}_q\Big|\F_t\right],\\
N_{a,b,t}^{q,\infty,\infty}:&=\displaystyle r_q(0)\mathbb{P}_x\left(t<\bm{e}_q<T_a\wedge T_b|\F_t\right),\\
\label{c62}
M_{a,b,t}^{q,\infty,\infty}:&=r_q(0)\mathbb{P}_x\left(\bm{e}_q<T_a\wedge T_b|\F_t\right)
\end{align}
for $q>0.$
\begin{thm}
\label{A1}
Let $x\in \R$. Then, $(M_{a,b,t}^{(0),\infty,\infty},\ t\ge 0)$ is a non-negative $((\F_t),\mathbb{P}_x)$-martingale, and it holds that 
\begin{align}
\label{c65}
\lim_{q\to 0+}N_{a,b,t}^{q,\infty,\infty}=\lim_{q\to 0+}M_{a,b,t}^{q,\infty,\infty}=M_{a,b,t}^{(0),\infty,\infty}\ \mathrm{a.s.\ and\ in}\ L^1(\mathbb{P}_x).
\end{align}
Consequently, if $M_{a,b,0}^{(0),\infty,\infty}>0$ under $\mathbb{P}_x$, it holds that 
\begin{align}
\label{c66}
\displaystyle \lim_{q\to 0+}\mathbb{P}_x[F_t|\ \bm{e}_q<T_a\wedge T_b]=\mathbb{P}_x\left[F_t \cdot\frac{M_{a,b,t}^{(0),\infty,\infty}}{M_{a,b,0}^{(0),\infty,\infty}}\right]
\end{align}
for all bounded $\F_t$-measurable functionals $F_t$.
\end{thm}
\begin{proof}
Note that
\begin{align}
\label{c67}
\lim_{\lambda_a,\lambda_b\to\infty}N_{a,b,t}^{q,\lambda_a,\lambda_b}=N_{a,b,t}^{q,\infty,\infty}\ \mathrm{a.s.}
\end{align}
We have by the lack of memory property of an exponential distribution and by the Markov property,
\begin{align*}
\label{c69}
\lim_{\lambda_a,\lambda_b\to \infty}N_{a,b,t}^{q,\lambda_a,\lambda_b}&=\lim_{\lambda_a,\lambda_b\to \infty}r_q(0)e^{-qt}e^{-\lambda_aL_t^a-\lambda_bL_t^b}\mathbb{P}_{X_t}\left[e^{-\lambda_aL_{\bm{e}_q}^a-\lambda_b L_{\bm{e}_q}^b}\right]\\
&=e^{-qt}r_q(0)\mathbb{P}_{X_t}\left[\int_0^{T_a\wedge T_b}qe^{-qs}ds\right]1_{\{t<T_a\wedge T_b\}}\ \mathrm{a.s.}
\stepcounter{equation}\tag{\theequation} 
\end{align*}
Therefore, we obtain by the equation (3.24) of \cite{ISY},
\begin{align*}
\label{c70}
\lim_{q\to 0+}N_{a,b,t}^{q,\infty,\infty}=\lim_{q\to 0+}e^{-qt}r_q(0)\mathbb{P}_{X_t}\left[\int_0^{T_a\wedge T_b}qe^{-qs}ds\right]1_{\{t<T_a\wedge T_b\}}=M_{a,b,t}^{(0),\infty,\infty}\ \mathrm{a.s.}
\stepcounter{equation}\tag{\theequation} 
\end{align*}
We have
\begin{align*}
M_{a,b,t}^{q,\infty,\infty}-N_{a,b,t}^{q,\infty,\infty}&=r_q(0)\P_x(\bm{e}_q<T_a\wedge T_b,\ \bm{e}_q\le t|\F_t)\\
&=r_q(0)1_{\{\bm{e}_q<T_a\wedge T_b\}}1_{\{\bm{e}_q\le t\}}\to 0\ \mathrm{a.s.}
\stepcounter{equation}\tag{\theequation} 
\end{align*}
as $q\to 0+$. Therefore, we obtain
\begin{align}
\lim_{q\to 0+}M_{a,b,t}^{q,\infty,\infty}=M_{a,b,t}^{(0),\infty,\infty}\ \mathrm{a.s.}
\end{align}
Finally, using Theorem 15.2 of \cite{Tukada}, the $L^1(\P_x)$ convergence holds.
\end{proof}

We define
\begin{align}
\label{d54}
N_{a,b,t}^{c,\lambda_a,\lambda_b}:&=\displaystyle h^B(c)\mathbb{P}_x\left[\Gamma_{a,b,T_c}^{\lambda_a,\lambda_b};\ t<T_c\Big|\F_t\right],\\
N_{a,b,t}^{c,\infty,\infty}:&=h^B(c)\displaystyle \mathbb{P}_{x}\left(t<T_c<T_a\wedge T_b|\F_t\right),\\
\label{d55}
M_{a,b,t}^{c,\infty,\infty}:&=h^B(c)\mathbb{P}_{x}\left(T_c<T_a\wedge T_b|\F_t\right)
\end{align}
for distinct points $a,b,c\in \R.$
\begin{thm}
\label{A2}
Let $x\in \R$. Then, $(M_{a,b,t}^{(\pm 1),\infty,\infty},\ t\ge 0)$ is a non-negative $((\F_t),\mathbb{P}_x)$-martingale, and it holds that 
\begin{align}
\label{d59}
\lim_{c\to \pm \infty}N_{a,b,t}^{c,\infty,\infty}=\lim_{c\to \pm \infty}M_{a,b,t}^{c,\infty,\infty}=M_{a,b,t}^{(\pm 1),\infty,\infty}\ \mathrm{a.s.\ and\ in}\ L^1(\mathbb{P}_x).
\end{align}
Consequently, if $M_{a,b,0}^{(\pm 1),\infty,\infty}>0$ under $\mathbb{P}_x$, it holds that
\begin{align}
\label{d60}
\lim_{c\to \pm \infty }\mathbb{P}_x[F_t|\ T_c<T_a\wedge T_b]=\mathbb{P}_x\left[F_t \cdot\frac{M_{a,b,t}^{(\pm 1),\infty,\infty}}{M_{a,b,0}^{(\pm 1),\infty,\infty}}\right]
\end{align}
for all bounded $\F_t$-measurable functionals $F_t$.
\end{thm}
\begin{proof}
Note that
\begin{align}
\lim_{\lambda_a,\lambda_b\to\infty}N_{a,b,t}^{c,\lambda_a,\lambda_b}=N_{a,b,t}^{c,\infty,\infty}\ \mathrm{a.s.}
\end{align}
By the strong Markov property, we have
\begin{align*}
\lim_{\lambda_a,\lambda_b\to\infty}N_{a,b,t}^{c,\lambda_a,\lambda_b}&=\lim_{\lambda_a,\lambda_b\to\infty}1_{\{t<T_c\}}h^B(c)e^{-\lambda_aL_t^a-\lambda_bL_t^b}\P_{X_t}\left[e^{-\lambda_aL_{T_c}^a-\lambda_bL_{T_c}^b}\right]\\
&=1_{\{t<T_c\}}h^B(c)1_{\{t<T_a\wedge T_b\}}\P_{X_t}(T_c<T_a\wedge T_b).
\stepcounter{equation}\tag{\theequation} 
\end{align*}
Therefore, we obtain by the equation (4.16) of \cite{ISY},
\begin{align}
\lim_{c\to \pm \infty}N_{a,b,t}^{c,\infty,\infty}=\lim_{c\to \pm \infty}1_{\{t<T_c\}}h^B(c)1_{\{t<T_a\wedge T_b\}}\P_{X_t}(T_c<T_a\wedge T_b)=M_{a,b,t}^{(\pm 1),\infty,\infty}\ \mathrm{a.s.}
\end{align}
The rest of the proof is the same as in Theorem \ref{A1}, so we omit it.
\end{proof}

We define
\begin{align}
N_{a,b,t}^{c,d,\lambda_a,\lambda_b}&:=\displaystyle h^C(c,-d)\mathbb{P}_x\left[\Gamma_{a,b,T_c\wedge T_{-d}}^{\lambda_a,\lambda_b};\ t<T_c\wedge T_{-d}\Big|\F_t\right],\\
\label{e32}
N_{a,b,t}^{c,d,\infty,\infty}&:=h^C(c,-d)\mathbb{P}_{x}\left(t<T_{c}\wedge T_{-d}<T_a\wedge T_b|\F_{t}\right),\\
\label{e33}
M_{a,b,t}^{c,d,\infty,\infty}&:=h^C(c,-d)\mathbb{P}_{x}\left(T_c\wedge T_{-d}<T_a\wedge T_b|\F_{t}\right)
\end{align}
for $c,d>0$.
\begin{thm}
\label{A3}
Let $x\in \R$ and $-1\le\gamma\le 1.$ Then, $(M_{a,b,t}^{(\gamma),\infty,\infty},\ t\ge 0)$ is a non-negative $((\F_t),\mathbb{P}_x)$-martingale, and it holds that 
\begin{align}
\label{e37}
\lim_{(c,d)\stackrel{(\gamma)}{\to}\infty}N_{a,b,t}^{c,d,\infty,\infty}=\lim_{(c,d)\stackrel{(\gamma)}{\to}\infty}M_{a,b,t}^{c,d,\infty,\infty}=M_{a,b,t}^{(\gamma),\infty,\infty}\ \mathrm{a.s.\ and\ in}\ L^1(\mathbb{P}_x).
\end{align}
Consequently, if $M_{a,b,0}^{(\gamma),\infty,\infty}>0$ under $\mathbb{P}_x$, it holds that 
\begin{align}
\label{e38}
\lim_{(c,d)\stackrel{(\gamma)}{\to}\infty}\mathbb{P}_x\left[F_t|\ T_c\wedge T_{-d}<T_a\wedge T_b\right]= \mathbb{P}_x\left[F_t\cdot \frac{M_{a,b,t}^{(\gamma),\infty,\infty}}{M_{a,b,0}^{(\gamma),\infty,\infty}}\right]
\end{align}
for all bounded $\F_t$-measurable functionals $F_t$.
\end{thm}
\begin{proof}
Note that
\begin{align}
\lim_{\lambda_a,\lambda_b\to\infty}N_{a,b,t}^{c,d,\lambda_a,\lambda_b}=N_{a,b,t}^{c,d,\infty,\infty}\ \mathrm{a.s.}
\end{align}
By the strong Markov property, we have
\begin{align*}
\lim_{\lambda_a,\lambda_b\to\infty}N_{a,b,t}^{c,d,\lambda_a,\lambda_b}&=\lim_{\lambda_a,\lambda_b\to\infty}1_{\{t<T_c\wedge T_{-d}\}}e^{-\lambda_a L_t^a-\lambda_b L_t^b}h^C(c,-d)\P_{X_t}\left[e^{-\lambda_a L_{T_c\wedge T_{-d}}^a-\lambda_b L_{T_c\wedge T_{-d}}^b}\right]\\
&=1_{\{t<T_c\wedge T_{-d}\}}1_{\{t<T_a\wedge T_b\}}h^C(c,-d)\P_{X_t}(T_c\wedge T_{-d}<T_a\wedge T_b).
  \stepcounter{equation}\tag{\theequation} 
\end{align*}
Therefore, we obtain by the equation (5.6) of \cite{ISY},
\begin{align*}
\lim_{(c,d)\stackrel{(\gamma)}{\to}\infty}N_{a,b,t}^{c,d,\infty,\infty}&=\lim_{(c,d)\stackrel{(\gamma)}{\to}\infty}1_{\{t<T_c\wedge T_{-d}\}}1_{\{t<T_a\wedge T_b\}}h^C(c,-d)\P_{X_t}(T_c\wedge T_{-d}<T_a\wedge T_b)\\
&=M_{a,b,t}^{(\gamma),\infty,\infty}\ \mathrm{a.s.}
  \stepcounter{equation}\tag{\theequation} 
\end{align*}
The rest of the proof is the same as in Theorem \ref{A1}, so we omit it.
\end{proof}

We define
\begin{align}
N_{a,b,t}^{c,u,\lambda_a,\lambda_b}:&=\displaystyle h^B(c)\mathbb{P}_x\left[\Gamma_{a,b,\eta_u^c}^{\lambda_a,\lambda_b};\ t<\eta_u^c\Big|\F_t\right],\\
\label{f29}
N_{a,b,t}^{c,u,\infty,\infty}:&=h^B(c)\mathbb{P}_x\left(t<\eta_{u}^c<T_a\wedge T_b|\F_{t}\right),\\
\label{f30}
M_{a,b,t}^{c,u,\infty,\infty}:&=h^B(c)\mathbb{P}_{x}\left(\eta_u^c<T_a\wedge T_b|\F_{t}\right)
\end{align}
for $c\in \R$ and $u>0$.
\begin{thm}
\label{A5}
Let $x\in \R$. Then, it holds that 
\begin{align}
\label{f32}
\lim_{c\to \pm\infty}N_{a,b,t}^{c,u,\infty,\infty}=\lim_{c\to \pm\infty}M_{a,b,t}^{c,u,\infty,\infty}=M_{a,b,t}^{(\pm 1),\infty,\infty}\ \mathrm{a.s.\ and\ in}\ L^1(\mathbb{P}_x).
\end{align}
Consequently, if $M_{a,b,0}^{(\pm 1),\infty,\infty}>0$ under $\mathbb{P}_x$, it holds that 
\begin{align*}
\label{f33}
\lim_{c\to \pm \infty}\mathbb{P}_x[F_t|\ \eta_u^c<T_a\wedge T_b]=\mathbb{P}_x\left[F_t\cdot \frac{M_{a,b,t}^{(\pm 1),\infty,\infty}}{M_{a,b,0}^{(\pm 1),\infty,\infty}}\right]
\stepcounter{equation}\tag{\theequation}
\end{align*}
for all bounded $\F_t$-measurable functionals $F_t$.
\end{thm}
\begin{proof}
Note that
\begin{align}
\lim_{\lambda_a,\lambda_b\to\infty}N_{a,b,t}^{c,u,\lambda_a,\lambda_b}=N_{a,b,t}^{c,u,\infty,\infty}\ \mathrm{a.s.}
\end{align}
By the strong Markov property, we have
\begin{align*}
\lim_{\lambda_a,\lambda_b\to\infty}N_{a,b,t}^{c,u,\lambda_a,\lambda_b}&=\lim_{\lambda_a,\lambda_b\to\infty}1_{\{t<T_c\wedge T_{-d}\}}e^{-\lambda_a L_t^a-\lambda_b L_t^b}h^B(c)\P_{X_t}\left[e^{-\lambda_a L_{\eta_u^c}^a-\lambda_b L_{\eta_u^c}^b}\right]\\
&=1_{\{t<\eta_u^c\}}1_{\{t<T_a\wedge T_b\}}h^B(c)\P_{X_t}(\eta_u^c<T_a\wedge T_b).
  \stepcounter{equation}\tag{\theequation} 
\end{align*}
Therefore, we obtain by the equation (6.8) and (6.17) of \cite{ISY},
\begin{align*}
\lim_{c\to \pm\infty}N_{a,b,t}^{c,u,\infty,\infty}&=\lim_{c\to \pm\infty}1_{\{t<\eta_u^c\}}1_{\{t<T_a\wedge T_b\}}h^B(c)\P_{X_t}(\eta_u^c<T_a\wedge T_b)\\
&=M_{a,b,t}^{(\pm 1),\infty,\infty}\ \mathrm{a.s.}
  \stepcounter{equation}\tag{\theequation} 
\end{align*}
The rest of the proof is the same as in Theorem \ref{A1}, so we omit it.
\end{proof}


\section{One-point local time penalization with conditioning to avoid another point}
\label{S4}
We define
\begin{align}
\label{c71}
N_{a,b,t}^{q,\lambda_a,\infty}:&=\displaystyle r_q(0)\mathbb{P}_x\left[e^{-\lambda_aL_{\bm{e}_q}^a};\ t<\bm{e}_q<T_b\Big|\F_t\right],\\
\label{c72}
M_{a,b,t}^{q,\lambda_a,\infty}:&=r_q(0)\mathbb{P}_x\left[e^{-\lambda_aL_{\bm{e}_q}^a};\ \bm{e}_q<T_b\Big|\F_t\right]
\end{align}
for $q>0.$
\begin{thm}
\label{B1}
Let $x\in \R$. Then, $(M_{a,b,t}^{(0),\lambda_a,\infty},\ t\ge 0)$ is a non-negative $((\F_t),\mathbb{P}_x)$-martingale, and it holds that 
\begin{align}
\label{c76}
\lim_{q\to 0+}N_{a,b,t}^{q,\lambda_a,\infty}=\lim_{q\to 0+}M_{a,b,t}^{q,\lambda_a,\infty}=M_{a,b,t}^{(0),\lambda_a,\infty}\ \mathrm{a.s.\ and\ in}\ L^1(\mathbb{P}_x).
\end{align}
Consequently, if $M_{a,b,0}^{(0),\lambda_a,\infty}>0$ under $\mathbb{P}_x$, it holds that 
\begin{align}
\label{c77}
\displaystyle \lim_{q\to 0+}\frac{\mathbb{P}_x[F_t\cdot e^{-\lambda_aL_{\bm{e}_q}^a};\ \bm{e}_q<T_b]}{\mathbb{P}_x[e^{-\lambda_aL_{\bm{e}_q}^a};\ \bm{e}_q<T_b]}= \mathbb{P}_x\left[F_t\cdot \frac{M_{a,b,t}^{(0),\lambda_a,\infty}}{M_{a,b,0}^{(0),\lambda_a,\infty}}\right]
\end{align}
for all bounded $\F_t$-measurable functionals $F_t$.
\end{thm}
The proof is almost the same as that of Theorem \ref{A1}, and so we omit it.

We define
\begin{align}
\label{d61}
N_{a,b,t}^{c,\lambda_a,\infty}:&=h^B(c)\displaystyle \mathbb{P}_{x}\left[e^{-\lambda_aL_{T_c}^a};\ t<T_c<T_b\Big|\F_t\right],\\
\label{d62}
M_{a,b,t}^{c,\lambda_a,\infty}:&=h^B(c)\mathbb{P}_{x}\left[e^{-\lambda_aL_{T_c}^a};\ T_c<T_b\Big|\F_t\right]
\end{align}
for distinct points $a,b,c\in \R.$
\begin{thm}
\label{B2}
Let $x\in \R$. Then, $(M_{a,b,t}^{(\pm 1),\lambda_a,\infty},\ t\ge 0)$ is a non-negative $((\F_t),\mathbb{P}_x)$-martingale, and it holds that 
\begin{align}
\label{d66}
\lim_{c\to \pm \infty}N_{a,b,t}^{c,\lambda_a,\infty}=\lim_{c\to \pm \infty}M_{a,b,t}^{c,\lambda_a,\infty}=M_{a,b,t}^{(\pm 1),\lambda_a,\infty}\ \mathrm{a.s.\ and\ in}\ L^1(\mathbb{P}_x).
\end{align}
Consequently, if $M_{a,b,0}^{(\pm 1),\lambda_a,\infty}>0$ under $\mathbb{P}_x$, it holds that
\begin{align}
\label{d67}
\lim_{c\to \pm \infty }\frac{\mathbb{P}_{x}[F_t\cdot e^{-\lambda_aL_{T_c}^a};\ T_c<T_b]}{\mathbb{P}_{x}[e^{-\lambda_aL_{T_c}^a};\ T_c<T_b]}=\mathbb{P}_x\left[F_t \cdot\frac{M_{a,b,t}^{(\pm 1),\lambda_a,\infty}}{M_{a,b,0}^{(\pm 1),\lambda_a,\infty}}\right]
\end{align}
for all bounded $\F_t$-measurable functionals $F_t$.
\end{thm}
The proof is almost the same as that of Theorem \ref{A1}, and so we omit it.

We define
\begin{align}
\label{e39}
N_{a,b,t}^{c,d,\lambda_a,\infty}&:=h^C(c,-d)\mathbb{P}_{x}\left[e^{-\lambda_aL_{T_c\wedge T_{-d}}^a};\ t<T_{c}\wedge T_{-d}<T_b\Big|\F_{t}\right],\\
\label{e40}
M_{a,b,t}^{c,d,\lambda_a,\infty}&:=h^C(c,-d)\mathbb{P}_{x}\left[e^{-\lambda_aL_{T_c\wedge T_{-d}}^a};\ T_{c}\wedge T_{-d}<T_b\Big|\F_{t}\right]
\end{align}
for $c,d>0$.
\begin{thm}
\label{B3}
Let $x\in \R$ and $-1\le\gamma\le1.$ Then, $(M_{a,b,t}^{(\gamma),\lambda_a,\infty},\ t\ge 0)$ is a non-negative $((\F_t),\mathbb{P}_x)$-martingale, and it holds that 
\begin{align}
\label{e44}
\lim_{(c,d)\stackrel{(\gamma)}{\to}\infty}N_{a,b,t}^{c,d,\lambda_a,\infty}=\lim_{(c,d)\stackrel{(\gamma)}{\to}\infty}M_{a,b,t}^{c,d,\lambda_a,\infty}=M_{a,b,t}^{(\gamma),\lambda_a,\infty}\ \mathrm{a.s.\ and\ in}\ L^1(\mathbb{P}_x).
\end{align}
Consequently, if $M_{a,b,0}^{(\gamma),\lambda_a,\infty}>0$ under $\mathbb{P}_x$, it holds that 
\begin{align}
\label{e45}
\lim_{(c,d)\stackrel{(\gamma)}{\to}\infty}\frac{\mathbb{P}_x[F_t\cdot e^{-\lambda_aL_{T_c\wedge T_{-d}}^a};\ T_{c}\wedge T_{-d}<T_b]}{\mathbb{P}_x[e^{-\lambda_aL_{T_c\wedge T_{-d}}^a};\ T_{c}\wedge T_{-d}<T_b]}= \mathbb{P}_x\left[F_t \cdot\frac{M_{a,b,t}^{(\gamma),\lambda_a,\infty}}{M_{a,b,0}^{(\gamma),\lambda_a,\infty}}\right]
\end{align}
for all bounded $\F_t$-measurable functionals $F_t$.
\end{thm}
The proof is almost the same as that of Theorem \ref{A1}, and so we omit it.

We define
\begin{align}
\label{f34}
N_{a,b,t}^{c,u,\lambda_a,\infty}:&=h^B(c)\mathbb{P}_x\left[e^{-\lambda_aL_{\eta_u^c}^a};\ t<\eta_{u}^c<T_b\Big|\F_{t}\right],\\
\label{f35}
M_{a,b,t}^{c,u,\lambda_a,\infty}:&=h^B(c)\mathbb{P}_{x}\left[e^{-\lambda_aL_{\eta_u^c}^a};\ \eta_{u}^c<T_b\Big|\F_{t}\right]
\end{align}
for $c\in \R$ and $u>0$.
\begin{thm}
\label{B5}
Let $x\in \R$. Then, it holds that 
\begin{align}
\label{f37}
\lim_{c\to \pm\infty}N_{a,b,t}^{c,u,\lambda_a,\infty}=\lim_{c\to \pm\infty}M_{a,b,t}^{c,u,\lambda_a,\infty}=M_{a,b,t}^{(\pm 1),\lambda_a,\infty}\ \mathrm{a.s.\ and\ in}\ L^1(\mathbb{P}_x).
\end{align}
Consequently, if $M_{a,b,0}^{(\pm 1),\lambda_a,\infty}>0$ under $\mathbb{P}_x$, it holds that 
\begin{align}
\label{f38}
\lim_{c\to \pm \infty}\frac{\mathbb{P}_x\left[F_t\cdot e^{-\lambda_aL_{\eta_u^c}^a};\ \eta_{u}^c<T_b\right]}{\mathbb{P}_x[e^{-\lambda_aL_{\eta_u^c}^a};\ \eta_{u}^c<T_b]}=\mathbb{P}_x\left[F_t \cdot\frac{M_{a,b,t}^{(\pm 1),\lambda_a,\infty}}{M_{a,b,0}^{(\pm 1),\lambda_a,\infty}}\right]
\end{align}
for all bounded $\F_t$-measurable functionals $F_t$.
\end{thm}
The proof is almost the same as that of Theorem \ref{A1}, and so we omit it.


\section{The inverse local time clock}
\label{S5}
In this section, we consider the inverse local time clock with the limit as $u$ tends to infinity and $c$ being fixed. We define
\begin{align}
\label{f39}
N_{a,b,t}^{u,\lambda_a,\lambda_b,c}:&=\frac{1}{\P_c\left[\Gamma_{a,b,\eta_{u}^c}^{\lambda_a,\lambda_b}\right]}\cdot \mathbb{P}_x\left[\Gamma_{a,b,\eta_u^c}^{\lambda_a,\lambda_b};\ t<\eta_{u}^c\Big|\F_{t}\right],\\
\label{f40}
M_{a,b,t}^{u,\lambda_a,\lambda_b,c}:&=\frac{1}{\P_c\left[\Gamma_{a,b,\eta_{u}^c}^{\lambda_a,\lambda_b}\right]}\cdot \mathbb{P}_x\left[\Gamma_{a,b,\eta_u^c}^{\lambda_a,\lambda_b}\Big|\F_{t}\right],
\end{align}
for $c\in \R$ and $u>0$. We further define
\begin{align}
\label{f41}
M_{a,b,t}^{\lambda_a,\lambda_b,c}:=e^{L_t^cH_{a,b}^{c,\lambda_a,\lambda_b}}\Gamma_{a,b,t}^{\lambda_a,\lambda_b},
\end{align}
where the constant $H_{a,b}^{c,\lambda_a,\lambda_b}$ is defined by (\ref{f47-1}). Note that the constant $H_{a,b}^{c,\lambda_a,\lambda_b}$ is independent of $u$.

\begin{thm}
\label{f42}
Let $x\in \R$. Then, $(M_{a,b,t}^{\lambda_a,\lambda_b,c},\ t\ge 0)$ is a non-negative $((\F_t),\mathbb{P}_x)$-martingale, and it holds that 
\begin{align}
\label{f43}
\lim_{u\to \infty}N_{a,b,t}^{u,\lambda_a,\lambda_b,c}=\lim_{u\to \infty}M_{a,b,t}^{u,\lambda_a,\lambda_b,c}=M_{a,b,t}^{\lambda_a,\lambda_b,c}\ \mathrm{a.s.\ and\ in}\ L^1(\mathbb{P}_x).
\end{align}
Consequently, if $M_{a,b,0}^{\lambda_a,\lambda_b,c}>0$ under $\mathbb{P}_x$, it holds that 
\begin{align}
\label{f44}
\lim_{u\to \infty}\frac{\mathbb{P}_x\left[F_t\cdot \Gamma_{a,b,\eta_u^c}^{\lambda_a,\lambda_b}\right]}{\mathbb{P}_x\left[\Gamma_{a,b,\eta_u^c}^{\lambda_a,\lambda_b}\right]}= \mathbb{P}_x\left[F_t \cdot\frac{M_{a,b,t}^{\lambda_a,\lambda_b,c}}{M_{a,b,0}^{\lambda_a,\lambda_b,c}}\right]
\end{align}
for all bounded $\F_t$-measurable functionals $F_t$.
\end{thm}
\begin{proof}
By definition of $\eta_u^c$, we have
\begin{align*}
\label{f45}
\eta_u^c&=\inf\{s:\ L_s^c>u\}\\
&=\inf\{s:\ L_s^c-L_t^c>u-v\}\Big|_{v=L_t^c}\\
&=\inf\{s:\ L_{s+t}^c-L_t^c>u-v\}\Big|_{v=L_t^c}+t\\
&=\inf \{s:\ L_s^c\circ \theta_t >u-v\}\Big|_{v=L_t^c}+t\\
&=\eta_{u-v}^c\circ \theta_t \Big|_{v=L_t^c}+t.
\stepcounter{equation}\tag{\theequation}
\end{align*}
By the Markov property at $t$, we have
\begin{align*}
\label{f46}
\mathbb{P}_x\left[\Gamma_{a,b,\eta_u^c}^{\lambda_a,\lambda_b};\ t<\eta_{u}^c\Big|\F_{t}\right]&=\mathbb{P}_x\left[e^{-\lambda_a(L_{\eta_{u-v}^c\circ \theta_t+t}^a-L_t^a+L_t^a)-\lambda_b (L_{\eta_{u-v}^c\circ \theta_t+t}^b-L_t^b+L_t^b)};\ t<\eta_{u}^c\Big|\F_{t}\right]\Big|_{v=L_t^c}\\
&=\mathbb{P}_x\left[e^{-\lambda_a(L_{\eta_{u-v}^c}^a\circ \theta_t+L_t^a)-\lambda_b (L_{\eta_{u-v}^c}^b\circ \theta_t+L_t^b)};\ t<\eta_{u}^c\Big|\F_{t}\right]\Big|_{v=L_t^c}\\
&=1_{\{t<\eta_u^c\}}e^{-\lambda_a L_t^a-\lambda_b L_t^b}\mathbb{P}_{X_t}\left[e^{-\lambda_a L_{\eta_{u-v}^c}^a-\lambda_b L_{\eta_{u-v}^c}^b}\right]\Big|_{v=L_t^c}.
\stepcounter{equation}\tag{\theequation}
\end{align*}
By (6.5) of \cite{ISY}, we can let
\begin{align}
\label{f47-1}
\P_c\left[\Gamma_{a,b,\eta_{u}^c}^{\lambda_a,\lambda_b}\right]=e^{-uH_{a,b}^{c,\lambda_a,\lambda_b}}
\end{align}
for $u\ge 0.$ Thus, we have
\begin{align*}
\label{f47}
\frac{1}{\P_c\left[\Gamma_{a,b,\eta_{u}^c}^{\lambda_a,\lambda_b}\right]}\cdot\mathbb{P}_c\left[e^{-\lambda_aL_{\eta_{u-v}^c}^a-\lambda_b L_{\eta_{u-v}^c}^b}\right]\Big|_{v=L_t^c}=e^{L_t^cH_{a,b}^{c,\lambda_a,\lambda_b}}.
\stepcounter{equation}\tag{\theequation}
\end{align*}
Therefore, we obtain
\begin{align*}
\label{f48}
\displaystyle\lim_{u\to \infty}N_{a,b,t}^{u,\lambda_a,\lambda_b,c}=M_{a,b,t}^{\lambda_a,\lambda_b,c}\ \mathrm{a.s.}
\stepcounter{equation}\tag{\theequation}
\end{align*}
Since we have
\begin{align*}
\label{f49}
\lim_{u\to \infty}\left(M_{a,b,t}^{u,\lambda_a,\lambda_b,c}-N_{a,b,t}^{u,\lambda_a,\lambda_b,c}\right)&=\lim_{u\to \infty}\frac{1}{\P_c\left[\Gamma_{a,b,\eta_{u}^c}^{\lambda_a,\lambda_b}\right]}\cdot\mathbb{P}_x\left[\Gamma_{a,b,\eta_u^c}^{\lambda_a,\lambda_b};\ \eta_{u}^c\le t\Big|\F_{t}\right]\\
&=\lim_{u\to \infty}\frac{\Gamma_{a,b,\eta_u^c}^{\lambda_a,\lambda_b}}{\P_c\left[\Gamma_{a,b,\eta_{u}^c}^{\lambda_a,\lambda_b}\right]}\cdot 1_{\{\eta_u^c\le t\}}=0\ \mathrm{a.s.},
\stepcounter{equation}\tag{\theequation}
\end{align*}
we obtain
\begin{align}
\label{f50}
\displaystyle\lim_{u\to \infty}M_{a,b,t}^{u,\lambda_a,\lambda_b,c}=M_{a,b,t}^{\lambda_a,\lambda_b,c}\ \mathrm{a.s.}
\end{align}

Next, we show that the convergence of (\ref{f48}) and (\ref{f50}) hold also in the sense of $L^1(\mathbb{P}_x).$ Since $M_{a,b,t}^{\lambda_a,\lambda_b,c}\in L^1(\mathbb{P}_x)$ for all $t>0,$ by the dominated convergence theorem, we obtain
\begin{align}
\label{f53}
\lim_{u\to \infty}N_{a,b,t}^{u,\lambda_a,\lambda_b,c}=M_{a,b,t}^{\lambda_a,\lambda_b,c}\ \mathrm{in}\ L^1(\mathbb{P}_x)
\end{align}
for all $t>0$. For $q>0$, we have by (\ref{b17}),
\begin{align}
\label{f54}
\mathbb{P}_x(\eta_u^c\le t)\le \mathbb{P}_c(\eta_u^c\le t)\le \mathbb{P}_c\left[e^{q(t-\eta_u^c)};\ \eta_u^c\le t\right]\le e^{qt}\mathbb{P}_c\left[e^{-q\eta_u^c}\right]=e^{qt}e^{-\frac{u}{r_q(0)}}. 
\end{align}
Since $r_q(0)\to 0$ as $q\to \infty$, we may take $q>0$ large enough so that $\frac{1}{r_q(0)}>H_{a,b}^{c,\lambda_a,\lambda_b}$. Thus, for such $q>0$, we have
\begin{align*}
\label{f55}
\mathbb{P}_x\left[M_{a,b,t}^{u,\lambda_a,\lambda_b,c}-N_{a,b,t}^{u,\lambda_a,\lambda_b,c}\right]&=\frac{1}{\P_c\left[\Gamma_{a,b,\eta_{u}^c}^{\lambda_a,\lambda_b}\right]}\cdot \mathbb{P}_x\left[\mathbb{P}_x\left[\Gamma_{\eta_u^c};\ \eta_{u}^c\le t|\F_{t}\right]\right]\\
&\le \frac{\mathbb{P}_x(\eta_u^c\le t)}{\P_c\left[\Gamma_{a,b,\eta_{u}^c}^{\lambda_a,\lambda_b}\right]}\le e^{qt}e^{u(H_{a,b}^{c,\lambda_a,\lambda_b}-\frac{1}{r_q(0)})}\to 0
\stepcounter{equation}\tag{\theequation}
\end{align*}
as $u\to \infty$. This shows that for all $t>0$, we obtain
\begin{align}
\label{f56}
M_{a,b,t}^{u,\lambda_a,\lambda_b,c}-N_{a,b,t}^{u,\lambda_a,\lambda_b,c}\to 0\ \mathrm{in}\ L^1(\mathbb{P}_x),
\end{align}
which implies 
\begin{align}
\label{f57}
\displaystyle\lim_{u\to \infty}M_{a,b,t}^{u,\lambda_a,\lambda_b,c}=M_{a,b,t}^{\lambda_a,\lambda_b,c}\ \mathrm{in}\ L^1(\mathbb{P}_x)
\end{align}
for all $t>0$.
\end{proof}

We define
\begin{align}
\label{f59}
N_{a,b,t}^{u,\infty,\infty,c}:&=\frac{1}{\P_c\left[\Gamma_{a,b,\eta_u^c}^{\infty,\infty}\right]}\cdot \mathbb{P}_x\left(t<\eta_{u}^c<T_a\wedge T_b\Big|\F_{t}\right),\\
\label{f60}
M_{a,b,t}^{u,\infty,\infty,c}:&=\frac{1}{\P_c\left[\Gamma_{a,b,\eta_u^c}^{\infty,\infty}\right]}\cdot \mathbb{P}_x\left(\eta_u^c<T_a\wedge T_b\Big|\F_{t}\right)
\end{align}
for $c\in \R$ and $u>0$. We further define
\begin{align}
\label{f61}
M_{a,b,t}^{\infty,\infty,c}:=e^{L_t^cH_{a,b}^{c,\infty,\infty}}\Gamma_{a,b,t}^{\infty,\infty},
\end{align}
where $H_{a,b}^{c,\infty,\infty}:=\lim_{\lambda_a,\lambda_b\to \infty}H_{a,b}^{c,\lambda_a,\lambda_b}.$
\begin{thm}
\label{f62}
Let $x\in \R$. Then, $(M_{a,b,t}^{\infty,\infty,c},\ t\ge 0)$ is a non-negative $((\F_t),\mathbb{P}_x)$-martingale, and it holds that 
\begin{align}
\label{f63}
\lim_{u\to \infty}N_{a,b,t}^{u,\infty,\infty,c}=\lim_{u\to \infty}M_{a,b,t}^{u,\infty,\infty,c}=M_{a,b,t}^{\infty,\infty,c}\ \mathrm{a.s.\ and\ in}\ L^1(\mathbb{P}_x).
\end{align}
Consequently, if $M_{a,b,0}^{\infty,\infty,c}>0$ under $\mathbb{P}_x$, it holds that 
\begin{align*}
\label{f64}
\lim_{u\to \infty}\mathbb{P}_x[F_t|\ \eta_u^c<T_a\wedge T_b]= \mathbb{P}_x\left[F_t \cdot\frac{M_{a,b,t}^{\infty,\infty,c}}{M_{a,b,0}^{\infty,\infty,c}}\right]
\stepcounter{equation}\tag{\theequation}
\end{align*}
for all bounded $\F_t$-measurable functionals $F_t$.
\end{thm}
The proof is almost the same as that of Theorem \ref{A1}, and so we omit it.

We define
\begin{align}
\label{f66}
N_{a,b,t}^{u,\infty,\infty,c}:&=\frac{1}{\P_c\left[\Gamma_{a,b,\eta_u^c}^{\lambda_a,\infty}\right]}\cdot \mathbb{P}_x\left[e^{-\lambda_aL_{\eta_u^c}^a};\ t<\eta_{u}^c<T_b\Big|\F_{t}\right],\\
\label{f67}
M_{a,b,t}^{u,\infty,\infty,c}:&=\frac{1}{\P_c\left[\Gamma_{a,b,\eta_u^c}^{\lambda_a,\infty}\right]}\cdot \mathbb{P}_x\left[e^{-\lambda_aL_{\eta_u^c}^a};\ \eta_{u}^c<T_b\Big|\F_{t}\right]
\end{align}
for $c\in \R$ and $u>0$. We further define
\begin{align}
\label{f68}
M_{a,b,t}^{\infty,\infty,c}:=e^{L_t^cH_{a,b}^{c,\lambda_a,\infty}}\Gamma_{a,b,t}^{\lambda_a,\infty},
\end{align}
where $H_{a,b}^{c,\lambda_a,\infty}:=\lim_{\lambda_b\to \infty}H_{a,b}^{c,\lambda_a,\lambda_b}.$
\begin{thm}
\label{f69}
Let $x\in \R$. Then, $(M_{a,b,t}^{\infty,\infty,c},\ t\ge 0)$ is a non-negative $((\F_t),\mathbb{P}_x)$-martingale, and it holds that 
\begin{align}
\label{f70}
\lim_{u\to \infty}N_{a,b,t}^{u,\infty,\infty,c}=\lim_{u\to \infty}M_{a,b,t}^{u,\infty,\infty,c}=M_{a,b,t}^{\infty,\infty,c}\ \mathrm{a.s.\ and\ in}\ L^1(\mathbb{P}_x).
\end{align}
Consequently, if $M_{a,b,0}^{\infty,\infty,c}>0$ under $\mathbb{P}_x$, it holds that 
\begin{align}
\label{f71}
\lim_{u\to \infty}\frac{\mathbb{P}_x[F_t\cdot e^{-\lambda_aL_{\eta_u^c}^a};\ \eta_{u}^c<T_b]}{\mathbb{P}_x[e^{-\lambda_aL_{\eta_u^c}^a};\ \eta_{u}^c<T_b]}= \mathbb{P}_x\left[F_t\cdot \frac{M_{a,b,t}^{\infty,\infty,c}}{M_{a,b,0}^{\infty,\infty,c}}\right]
\end{align}
for all bounded $\F_t$-measurable functionals $F_t$.
\end{thm}
The proof is almost the same as that of Theorem \ref{A1}, and so we omit it.


\section{Corrections to some results of Takeda--Yano \cite{TY}}
\label{S6}
There is an error in the assertion of Lemma 6.1 of \cite{TY}. The correct assertion is as follows:
\begin{prop}[Lemma 6.1 of \cite{TY},\ corrected]
For distinct points $a\neq b$, it holds that
\begin{align*}
h^C(a,b):&=\mathbb{P}_0[L_{T_a\wedge T_b}^0]\\
&=\frac{1}{h^B(a-b)}\left\{\begin{aligned}
&(h(b)+h(-a))h(a-b)+(h(a)+h(-b))h(b-a)\\
&\qquad +(h(a)-h(b))(h(-b)-h(-a))-h(a-b)h(b-a)
\end{aligned}\right\}.
\stepcounter{equation}\tag{\theequation}
\end{align*}
\end{prop}
\begin{proof}
For $q>0$, by the strong Markov property, we have
\begin{align*}
\P\left[\int_0^\infty e^{-qt}dL_t\right]&=\P\left[\int_0^{T_a\wedge T_b}e^{-qt}dL_t\right]\\
&\qquad +\P[e^{-qT_a};\ T_a<T_b]\P_a\left[\int_0^\infty e^{-qt}dL_t\right]\\
&\qquad +\P[e^{-qT_b};\ T_b<T_a]\P_b\left[\int_0^\infty e^{-qt}dL_t\right].
\stepcounter{equation}\tag{\theequation}
\end{align*}
Using (\ref{b15}), we have
\begin{align}
r_q(0)=\P\left[\int_0^{T_a\wedge T_b}e^{-qt}dL_t\right]+\P[e^{-qT_a};\ T_a<T_b]r_q(-a)+\P[e^{-qT_b};\ T_b<T_a]r_q(-b).
\end{align}
Thus, by Lemma 3.5 of \cite{TY}, we have
\begin{align*}
\label{6.1}
\P&\left[\int_0^{T_a\wedge T_b}e^{-qt}dL_t\right]\\
&=r_q(0)-\P[e^{-qT_a};\ T_a<T_b]r_q(-a)-\P[e^{-qT_b};\ T_b<T_a]r_q(-b)\\
&=r_q(0)-\frac{h_q(b-a)+h_q(-b)-h_q(-a)-\frac{h_q(-b)h_q(b-a)}{r_q(0)}}{h_q^B(a-b)}\cdot r_q(-a)\\
&\qquad -\frac{h_q(a-b)+h_q(-a)-h_q(-b)-\frac{h_q(-a)h_q(a-b)}{r_q(0)}}{h_q^B(a-b)}\cdot r_q(-b).
\stepcounter{equation}\tag{\theequation}
\end{align*}
We consider only the numerator of the right-hand side. Then, we have
\begin{align*}
r_q(0)&h_q^B(a-b)-\left(h_q(b-a)+h_q(-b)-h_q(-a)-\frac{h_q(-b)h_q(b-a)}{r_q(0)}\right)r_q(-a)\\
&\qquad -\left(h_q(a-b)+h_q(-a)-h_q(-b)-\frac{h_q(-a)h_q(a-b)}{r_q(0)}\right)r_q(-b)\\
&=h_q(a-b)(r_q(0)-r_q(-b))+h_q(b-a)(r_q(0)-r_q(-a))-h_q(a-b)h_q(b-a)\\
&\qquad +h_q(-a)(r_q(-a)-r_q(-b))+h_q(-b)(r_q(-b)-r_q(-a))\\
&\qquad +\frac{h_q(-b)h_q(b-a)r_q(-a)}{r_q(0)}+\frac{h_q(-a)h_q(a-b)r_q(-b)}{r_q(0)}\\
&=h_q(a-b)h_q(b)+h_q(b-a)h_q(a)-h_q(a-b)h_q(b-a)\\
&\qquad +h_q(-a)(h_q(b)-h_q(a))+h_q(-b)(h_q(a)-h_q(b))\\
&\qquad +(h_q(-b)h_q(b-a))\mathbb{P}_a\left[e^{-qT_0}\right]+(h_q(-a)h_q(a-b))\mathbb{P}_b\left[e^{-qT_0}\right]\\
&\to h(a-b)h(b)+h(b-a)h(a)-h(a-b)h(b-a)\\
&\qquad +h(-a)(h(b)-h(a))+h(-b)(h(a)-h(b))\\
&\qquad +h(-b)h(b-a)+h(-a)h(a-b)\\
&=(h(b)+h(-a))h(a-b)+(h(a)+h(-b))h(b-a)\\
&\qquad +(h(a)-h(b))(h(-b)-h(-a))-h(a-b)h(b-a)
\end{align*}
as $q\to 0+.$ Therefore, letting $q\to 0+$ in the equation (\ref{6.1}), we obtain the assertion. 
\end{proof}

Next, there are some errors in the assertion of Theorem 1.7 of \cite{TY}. The correct assertion is as follows ($(a,b)\stackrel{\gamma}{\to}\infty$ is replaced with $(a,b)\stackrel{(\gamma)}{\to}\infty$):
\begin{thm}[Theorem 1.7 of \cite{TY}, corrected]
Suppose that the condition (A) is satisfied. Let $f$ be a positive integrable function, $x\in \R$, and $a,b>0$. Define
\begin{align}
N_t^{a,b}:&=h^C(a,-b)\P_x[f(L_{T_a\wedge T_{-b}}^0),\ t<T_a\wedge T_{-b}|\F_t],\\
M_t^{a,b}:&=h^C(a,-b)\P_x[f(L_{T_a\wedge T_{-b}}^0)|\F_t].
\end{align}
Then, it holds that
\begin{align*}
\lim_{(a,b)\stackrel{(\gamma)}{\to}\infty}N_t^{a,b}=\lim_{(a,b)\stackrel{(\gamma)}{\to}\infty}M_t^{a,b}=M_t^{(\gamma)}
\end{align*}
in the sense of $\P_x$-a.s. and in $L^1(\P_x),$ where the process $(M_t^{(\gamma)})_{t\ge 0}$ is defined by (1.8) of \cite{TY}. Consequently, if $M_0^{(\gamma)}>0$ under $\P_x$, it holds that
\begin{align}
\lim_{(a,b)\stackrel{(\gamma)}{\to}\infty}\frac{\P_x[F_t\cdot f(L_{T_a\wedge T_{-b}}^0)]}{\P_x[f(L_{T_a\wedge T_{-b}}^0)]}=\P_x\left[F_t\cdot \frac{M_t^{(\gamma)}}{M_0^{(\gamma)}}\right]
\end{align}
for all bounded $\F_t$-measurable functionals $F_t.$
\end{thm}
We omit the proof of this theorem because the proof of this theorem in \cite{TY} is correct.

\bibliographystyle{plain}

\end{document}